\documentclass{article}
\usepackage[margin=30truemm]{geometry}
\usepackage{graphicx} % Required for inserting images
\usepackage{cite}
\usepackage[pdfencoding=auto, psdextra]{hyperref}
\usepackage{amsmath,amsthm, amssymb,amsfonts}
\usepackage{algorithmic}
\usepackage{textcomp}
\usepackage{xcolor}
\usepackage{appendix}

\newtheorem{theorem}{Theorem}%[section]
\newtheorem{example}[theorem]{Example}
\newtheorem{lemma}[theorem]{Lemma}

\newtheorem{conjecture}[theorem]{Conjecture}
\newtheorem{definition}[theorem]{Definition}
\newtheorem{remark}[theorem]{Remark}
\numberwithin{equation}{section}

\title{On cyclotomic nearly-doubly-regular tournaments}
\author{Shohei Satake
\thanks{
Research and Education Institute for Semiconductors and Informatics, Kumamoto University\\
2-39-1, Kurokami, Chuo, Kumamoto, 860-8555, Japan.\\
E-mail: shohei-satake@kumamoto-u.ac.jp
}
}
\date{}

\begin{document}

\maketitle
\begin{abstract}
Nearly-doubly-regular tournaments have played significant roles in extremal graph theory. 
In this note, we construct new cyclotomic nearly-doubly-regular tournaments and determine their spectrum by establishing a new connection between cyclotomic nearly-doubly-regular tournaments and almost difference sets from combinatorial design theory.
Furthermore, under the celebrated Hardy-Littlewood conjecture F in analytic number theory, our results confirm the conjecture due to Sergey Savchenko (J. Graph Theory {\bf 83} (2016), 44--77) on the existence of infinitely many nearly-doubly-regular tournaments with the canonical spectrum.
\end{abstract}

\section{Introduction} 
A {\it tournament} is an orientation of a complete graph.
A tournament with $n$ vertices is said to be {\it regular} if $n$ is odd and out-degree of each vertex is $(n-1)/2$.
For even $n$, a tournament with $n$ vertices is said to be {\it near-regular} if out-degree of each vertex is equal to $n/2$ or $n/2-1$.
A regular tournament with $n$ vertices is a {\it doubly-regular tournament}, denoted by $DR_n$, if $n\equiv 3 \pmod{4}$ and in- and out-neighbors of each vertex induce regular tournaments.
Similarly, a regular tournament with $n$ vertices is a {\it nearly-doubly-regular tournament}, denoted by $NDR_n$, if $n\equiv 1 \pmod{4}$ and in- and out-neighbors of each vertex induce near-regular tournaments.
% A tournament is said to be {\it homogeneous} if out-neighbours of any vertex induce a regular tournament. Similarly, a tournament is said to be {\it nearly-doubly-regular} if out-neighbours of any vertex induce a near-regular tournament. 
Doubly-regular and near-doubly-regular tournaments play key roles in many branches of extremal graph theory such as the counting problem for cycles of a given length in tournaments (see e.g. \cite{AT82}, \cite{BH65}, \cite{C64}, \cite{GKLV23}, \cite{S16}, \cite{S17}, \cite{S24a}, \cite{S24b}). 
Indeed it is known that those minimizes the number of $4$-cycles and maximizes the number of $5$-cycles among regular tournaments (\cite{AT82}, \cite{S16}).
Furthermore, combining Theorem~5 in \cite{AT82} and Theorem~1 in \cite{CG91}, it is readily seen that these tournaments posses quasi-randomness (see also \cite{KS13}, \cite{S2019}, \cite{S2021}).
Hence constructing doubly-regular or nearly-doubly-regular tournaments is an interesting but challenging open problem.
In particular, there is only few known constructions for nearly-doubly-regular tournaments while more constructions are known for doubly-regular tournaments (see e.g. \cite{S2019}, \cite{LW01}).

In this note, we focus on cyclotomic tournaments introduced by Asti\'{e}-Vidal~\cite{A75}. 
% Cyclotomic tournaments produce douly-regular or nearly-doubly-regular tournaments. 
For a prime power $q$, let $\mathbb{F}_q$ denote the finite field of order $q$ and let $\mathbb{F}_q^+$ and $\mathbb{F}_q^*=\mathbb{F}_q\setminus \{0\}$ be the additive and multiplicative groups of $\mathbb{F}_q$, respectively.
Let $g$ be a primitive element of $\mathbb{F}_q$, that is, a generator of $\mathbb{F}_q^*$ and $C_0^{(2^{\ell})}$ denote the subgroup of $\mathbb{F}_q^*$ of index $2^{\ell}$. 
If $q-1=2^{\ell}m$, where $\ell\geq 1$ and $m$ is odd, then {\it the cyclotomic tournament $CT_{q}$ with $q$ vertices} is defined as the tournament with vertex set $\mathbb{F}_q$ in which there exists an arc from $x$ to $y$ if and only if $x-y \in C_0^{(2^{\ell})} \cup g C_0^{(2^{\ell})}\cup \cdots \cup g^{2^{\ell-1}-1}C_0^{(2^{\ell})}$. 
This is well-defined since $(D, -D)$ is a partition of $\mathbb{F}_q^*$ where $D=C_0^{(2^{\ell})} \cup gC_0^{(2^{\ell})} \cup \cdots \cup g^{2^{\ell-1}-1}C_0^{(2^{\ell})}$ and $-D=\{-d \mid d \in D\}$.
Note that if $\ell=1$, or equivalently, $q \equiv 3 \pmod{4}$, then $CT_q$ is exactly the {\it Paley tournament} (a.k.a. {\it quadratic residue tournament}) with $q$ vertices.
On the other hand, for the case that $\ell=2$, it has been conjectured (\cite{AD92}, \cite{S16}) that there exist infinitely many primes $p \equiv 5 \pmod{8}$ such that $CT_p$ is nearly-doubly-regular while the authors of \cite{AD92} and \cite{T80} verified that $CT_p$ is nearly-doubly-regular when $p=5, 13, 29, 53, 173, 229, 293, 733$ by computer search. 

In this note, we first prove the following theorem by establishing a new connection between cyclotomic nearly-doubly-regular tournaments and almost difference sets (see Definition~\ref{def-ADS}) from combinatorial design theory.

\begin{theorem}
\label{theorem-main1}
Let $q\equiv 5 \pmod{8}$ be a prime power (equivalently, $\ell=2$).
Then $CT_q$ is nearly-doubly-regular if and only if $q=s^2+4$ for some odd integer $s$.
\end{theorem}
Indeed if $q$ is a prime power so that $q=s^2+4$ for some odd integer $s$, then $q$ must be either a prime or $q=5^3$ (see Remark~\ref{rem:primepower}).

Moreover we determine the spectrum of the nearly-doubly-regular tournament $CT_q$.
It is worth noting that by the following theorem we obtain nearly-doubly-regular tournaments with the {\it canonical spectrum} that asymptotically maximize the number of $6$-cycles among regular tournaments (see Proposition~3 in \cite{S16}).  

\begin{theorem}
\label{theorem-main2}
Let $q\equiv 5 \pmod{8}$ be a prime power such that $q=s^2+4$ for some odd integer $s$.
Then all eigenvalues of the adjacency matrix of $CT_q$ are
\begin{equation}
\frac{q-1}{2},
    \frac{-1 \pm i\sqrt{q-2\sqrt{q}}}{2},
    \frac{-1 \pm i\sqrt{q+2\sqrt{q}}}{2},
\end{equation}
where the first eigenvalue has multiplicity $1$ and each of the other four eigenvalues has multiplicity $\frac{q-1}{4}$.
\end{theorem}

Remarkably, assuming the following widely-believed conjecture (e.g. \cite[Section A1]{G04}), Theorems~\ref{theorem-main1} and \ref{theorem-main2} imply that there exist infinitely many primes $p \equiv 5 \pmod{8}$ such that $CT_p$ is a nearly-doubly-regular tournament with the canonical spectrum, which also confirms the conjecture due to Sergey Savchenko on the existence of infinitely many such tournaments (see p.68 in \cite{S16}).

\begin{conjecture}[The Hardy-Littlewood conjecture F, \cite{HL23}]
\label{conj:HL}
Suppose that $a>0, b, c$ are integers such that $(a, b, c)=1$, either $a+b$ or $c$ is odd, and $b^2-4ac$ is not a square.
Then there exist infinitely many positive integers $m$ such that $am^2+bm+c$ is a prime.
\end{conjecture}
Notice that the quadratic polynomial 
$s^2+4$
satisfies the condition in Conjecture~\ref{conj:HL}.
Our computer search using Mathematica found $455,927$ primes of the form $s^2+4$ when $1 \leq s \leq 10^7$.
The list of these primes contains seven primes $13$, $29$, $53$, $173$, $229$, $293$, $733$ which are exactly the orders of cyclotomic nearly-doubly-regular tournaments found in \cite{AD92} and \cite{T80}. 

The rest of this note is organized as follows. Sections~\ref{sect-thm1} and \ref{sect-thm2} present the proofs of Theorems~\ref{theorem-main1} and \ref{theorem-main2}, respectively. Finally, some concluding remarks are given in Section~\ref{sect-rem}.

% It should be noted that Hardy and Littlewood~\cite{HL23} also made a conjecture on an asymptotic formula of the counting function $P(x)$ of such primes less than $x$. 
% Numerical results in Appendix shows that their formula might almost match to the behavior of $P(x)$.

% Finally, we investigate the symmetry of cyclotomic tournaments.
% In \cite{AD92}, Asti\'{e}-Vidal and Dugat characterized vertex-symmetric tournaments with exactly two arc orbits.
% In particular, if $\ell=2$ (and hence $p \equiv 5 \pmod{8}$) and $CT_q$ is nearly-doubly-regular, then it is the unique vertex-transitive nearly-doubly-regular tournament with $q$ vertices and exactly two arc orbits. 

% We determine the automorphism group ${\rm Aut}(CT_p)$ for primes $p \equiv 5 \pmod{8}$.
% \begin{theorem}
% \label{theorem-main3}
% Let $p \equiv 5 \pmod{8}$ be a prime.
% For $a \in \mathbb{F}_p^*$, $b \in \mathbb{F}_p$, let $\sigma_{a, b}$ be a permutation on $\mathbb{F}_p$ that $x \mapsto ax+b$.
% Then it holds that
% \begin{equation}
%     {\rm Aut}(CT_p)=\{\sigma_{a, b} \mid a \in C_0^{(4)}, b \in \mathbb{F}_p \}.
% \end{equation}
% In particular, we have $\displaystyle |{\Aut}(CT_p)|=\frac{1}{2}\binom{p}{2}$.
% \end{theorem}

% we determine the automorphism group ${\rm Aut}(CT_p)$ for primes $p \equiv 5 \pmod{8}$.
% \begin{corollary}
% Let $p \equiv 5 \pmod{8}$ be a prime. Then $CT_p$ has exactly two arc-orbits, namely,
% $\{(x,y) \mid x-y \in C^{(4)}_{0}\}$ and $\{(x,y) \mid x-y \in C^{(4)}_{1}\}$.
% \end{corollary}

\section{Proof of Theorem~\ref{theorem-main1}}
\label{sect-thm1}
Our proof of Theorem~\ref{theorem-main1} is based on several results in combinatorial design theory.
First we introduce the definition of almost difference sets.

\begin{definition}[Almost difference set]
\label{def-ADS}
Let $n$ be a positive odd integer and let $k$ and $\lambda$ be positive integers. 
Let $G$ be an abelian group of order $n$.
Then a subset $D \subset G \setminus \{0\}$ with $k$ elements is called an {\it $(n, k, \lambda)$-almost difference set} of $G$ if for
some $(n-1)/2$ non-zero elements $a \in G$, the equation
\[
x-y=a
\]
has exactly $\lambda$ solutions $(x, y)\in D \times D$, and for other non-zero elements, the equation has exactly $\lambda+1$ solutions $(x, y)\in D \times D$. 
\end{definition}

To build tournaments from almost difference sets, we introduce {\it Cayley digraphs}.

\begin{definition}[Cayley digraph]
Let $G$ be an abelian group of order $n$ and $D \subset G \setminus \{0\}$ be a subset with $(n-1)/2$ elements such that $(D, -D)$ is a partition of $G \setminus \{0\}$. 
Then the {\it Cayley digraph} $Cay(G,D)$ is a digraph with vertex set $D$ in which for $x, y \in G$ there is an arc $(x, y)$ from $x$ to $y$ if and only if $x-y \in D$.
\end{definition}

Now we prove the following two lemmas for the proof of Theorem~\ref{theorem-main1}.

\begin{lemma}
\label{lem-ND1}
Let $n\equiv 1 \pmod{4}$ be a positive integer. 
Then a regular tournament $T$ with $n$ vertices is an $NDR_n$ if for each pair of two distinct vertices $u$ and $v$, the size of common out-neighbor of $u$ and $v$, denoted by $N_T^-(u,v)$, is $(n-1)/4$ or $(n-5)/4$.
\end{lemma}

\begin{proof}
It immediately follows from the definition of $NDR_n$.    
\end{proof}

\begin{lemma}
\label{lem-ND2}
Let $n \equiv 1 \pmod{4}$ be a positive integer.
Let $G$ be an abelian group of order $n$ and $D \subset G \setminus \{0\}$ be a subset with $(n-1)/2$ elements such that $(D, -D)$ is a partition of $G \setminus \{0\}$. 
Then the Cayley digraph $Cay(G,D)$ is an $NDR_n$ if and only if $D$ is an $(n, (n-1)/2, (n-5)/4)$-almost difference set of $G$.
\end{lemma}
\begin{proof}
Let $u$ and $v$ be distinct vertices. 
Since out-neighbor of $u$ is $u+D=\{u+d \mid d \in D\}$, 
it clearly holds that $N_{Cay(G, D)}^-(u, v)=(u+D) \cap (v+D)$.
On the other hand, for each $z \in (u+D) \cap (v+D)$, there uniquely exists $(d_1, d_2) \in D \times D$ such that $z=u+d_1=v+d_2$, which implies that there exists a bijection from $(u+D) \cap (v+D)$ to the set of solutions $(x, y) \in D \times D$ of the equation $x-y=u-v$.
The proposition follows from the above observation and Lemma~\ref{lem-ND1}.
\end{proof}

The following is a key theorem to prove Theorem~\ref{theorem-main1}.

\begin{theorem}[\cite{DHL99}, \cite{N14}]
\label{thm:DHL}
Let $q \equiv 5 \pmod{8}$ be a prime power.
%Let $g$ be a primitive element of the finite field $\mathbb{F}_p$ with $p$ elements. 
%Let $H$ be the multiplicative subgroup of $\mathbb{F}_p^*$ of index $4$.
Recall that $C_0^{(4)}$ is the multiplicative subgroup of $\mathbb{F}_q^*$ of index $4$.
For $i=1, 2, 3$, let 
$C_i^{(4)}=g^{i}C_0^{(4)}$.
%$C_i^{(4)}=\{g^j \mid j \equiv i \pmod{4}\}$.
Then for each $i$, $C^{(4)}_i \cup C^{(4)}_{i+1}$ is an $(q, (q-1)/2, (q-5)/4)$-almost difference set of the additive group $\mathbb{F}_q^+$ of $\mathbb{F}_q$ if and only if $q=s^2+4$ for some odd integer $s$. %with $s \equiv 1 \pmod{4}$.
\end{theorem}

Now we are ready to prove Theorem~\ref{theorem-main1}

\begin{proof}[Theorem~\ref{theorem-main1}]
Let $D=C^{(4)}_{0}\cup C^{(4)}_{1}$. Then it is not difficult to check that $(D, -D)$ is a partition of $\mathbb{F}_q^+ \setminus \{0\}$.
Hence it is readily seen that the cyclotomic tournament $CT_q$ with $\ell=2$ is $Cay(\mathbb{F}_q^+, D)$.
Hence the theorem follows from Lemma~\ref{lem-ND2} and Theorem~\ref{thm:DHL}.
\end{proof}

\begin{remark}
$Cay(\mathbb{F}_q^+, C^{(4)}_{0}\cup C^{(4)}_{1})$ and $Cay(\mathbb{F}_q^+, C^{(4)}_{i}\cup C^{(4)}_{i+1})$ are isomorphic for every $i=1,2,3$ (where each index is calculated modulo $4$).
In fact, it clearly holds that $C^{(4)}_{i} \cup C^{(4)}_{i+1}=g^i(C^{(4)}_{0}\cup C^{(4)}_{1})$ and thus the map $x \mapsto g^ix$ is an isomorphism. 
\end{remark}

\begin{example}
\label{eg:1}
Let $q=13$ and take $2 \in \mathbb{F}_{13}^*$ as a primitive element. Then it is not so difficult to check that $C^{(4)}_{0}\cup C^{(4)}_{1}=\{1, 2, 3, 5, 6, 9\}$ is a $(13, 6, 2)$-almost difference set of $\mathbb{F}_{13}^+$.
As show in \cite{T80} and \cite{AD92}, $CT_{13}=Cay(\mathbb{F}_{13}^+, C^{(4)}_{0}\cup C^{(4)}_{1})$ is an $NDR_{13}$.
Next let $q=29$ and take $2 \in \mathbb{F}_{29}^*$ as a primitive element. Then $C^{(4)}_{0}\cup C^{(4)}_{1}=\{1, 2, 3, 7, 11, 14, 16, 17, 19, 20, 21, 23, 24, 25\}$ is a $(29, 14, 6)$-almost difference set of $\mathbb{F}_{29}^+$.
Hence $CT_{29}=Cay(\mathbb{F}_{29}^+, C^{(4)}_{0}\cup C^{(4)}_{1})$ is an $NDR_{29}$ found in \cite{AD92}.
\end{example}

\section{Proof of Theorem~\ref{theorem-main2}}
\label{sect-thm2}

We present the proof of Theorem~\ref{theorem-main2}.
The following lemmas are useful to prove Theorem~\ref{theorem-main2}.

\begin{lemma}[e.g. \cite{N14}]
\label{lem:cha}
Let $G$ be an abelian group of order $n$ and $D$ an $(n, k, \lambda)$-almost difference set of $G$.
Let $S_D$ be the set of elements $a \in G \setminus \{0\}$ such that there exist $\lambda$ solutions $(x, y) \in D \times D$ of the equation $x-y=a$.
Then it holds for each non-trivial character $\chi$ of $G$ that
\begin{equation}
   \biggl|\sum_{d \in D}\chi(d)\biggr|^2=(k-\lambda-1)-\sum_{s\in S_D}\chi(s).
\end{equation}
\end{lemma}

\begin{lemma}
\label{lem:techmain}
Let $p \equiv 5 \pmod{8}$ be a prime such that $p=s^2+4$. %with $s \equiv 1 \pmod{4}$. 
Let $D=C_0^{(4)} \cup C_1^{(4)}$. 
Then it holds that $S_{D}$ is $C_0^{(2)}$ or $C_1^{(2)}$ where $C_0^{(2)}$ is the multiplicative subgroup of index $2$ and $C_1^{(2)}=g C_0^{(2)}$, $g$ is a primitive element of $\mathbb{F}_p$. 
% $C_i^{(2)}=\{g^j \mid j \equiv i \pmod{2}\}$ for $i=0, 1$.
\end{lemma}

\begin{proof}
For each $a \in \mathbb{F}_p^*$, let $M_a=\{(x, y) \in D \times D \mid x-y=a\}$.
By the definition of $C^{(4)}_i$, it holds for each $\alpha \in C_0^{(4)}$ that $\alpha D=D$, which implies that $|M_b|=|M_c|$ for every pair of distinct $b, c \in C^{(4)}_0$.
It also holds for each $\alpha \in C_2^{(4)}$ that $\alpha D=-D$.
Since $|M_a|=|M_{-a}|$ for each $a \in \mathbb{F}_p^*$, it is shown that $|M_b|=|M_c|$ for every pair of distinct $b, c \in C^{(4)}_0 \cup C^{(4)}_2$.
The lemma is obtained by the above observation, together with Theorem~\ref{thm:DHL} and simple facts that $C^{(4)}_0 \cup C^{(4)}_2=C^{(2)}_0$ and $|C^{(2)}_0|=(p-1)/2$.
\end{proof}
\begin{remark}
Since $1 \in C_0^{(4)}$, $S_D=C_0^{(2)}$ holds if and only if $1 \in S_D$, which follows from the discussion in the proof of Lemma~\ref{lem:techmain}.
\end{remark}

\begin{example}
Suppose that $D$ is a $(13, 6, 2)$-almost difference set in Example~\ref{eg:1}. 
One can easily check that $S_{D}=\{2, 5, 6, 7, 8, 11\}=C_1^{(2)}$.
When $D$ is a $(29, 14, 6)$-almost difference set in Example~\ref{eg:1}, it also can be checked that $S_{D}=\{2, 3, 8, 10, 11, 12, 14, 15, 17, 18, 19, 21, 26, 27\}=C_1^{(2)}$.
\end{example}

Now we prove Theorem~\ref{theorem-main2}.

\begin{proof}[Theorem~\ref{theorem-main2}]
For the simplicity we present the proof for the case that $q$ is a prime $p \equiv 5 \pmod{8}$ that $p=s^2+4$ for some odd integer $s$.
Note that the same discussion holds for prime powers $q$ with the same conditions as well (see Remarks~\ref{rem:primepower} and \ref{rem-125}).

Let $D=C_0^{(4)} \cup C_1^{(4)}$ which is a $(p, (p-1)/2, (p-5)/4)$-almost difference set of the additive group $\mathbb{F}_p^+$ by Theorem~\ref{thm:DHL}.
Recall that $CT_p$ with $\ell=2$ is $Cay(\mathbb{F}_p^+, D)$ and the multiset of eigenvalues of (the adjacency matrix of) $Cay(\mathbb{F}_p^+, D)$ is 
$
\{\lambda_a=\sum_{d \in D}\psi_a(d) \mid a \in \mathbb{F}_p\},
$
where 
$\psi_a(x)=\exp(\frac{2\pi i}{p}(ax))$.
Note that $\lambda_0=(p-1)/2$, which is the Perron-eigenvalue.

From now on, assume that $a \neq 0$.
Since $C^{(2)}_1=gC^{(2)}_0$, it follows from Lemmas~\ref{lem:cha} and \ref{lem:techmain} that
\begin{equation}
\label{eq:eigen1}
    |\lambda_a|^2
    =\begin{cases}
    \frac{p-1}{4}-\sum_{s\in C_0^{(2)}}\psi_a(s) & S_D=C^{(2)}_0;\\
    \\
    \frac{p-1}{4}-\sum_{s\in C_0^{(2)}}\psi_{ag}(s) & S_D=C^{(2)}_1.
    \end{cases}
\end{equation}
Below the proof for the case that $S_D=C^{(2)}_0$ will be noted since by (\ref{eq:eigen1}), the same discussion works for the another case.

First a simple calculation shows that
\begin{equation}
\label{eq:eigen2}
    \sum_{s\in C_0^{(2)}}\psi_a(s)=\frac{1}{2}\Biggl(\sum_{s \in \mathbb{F}_p}\psi_a(s^2)-1\Biggr),
\end{equation}
where $\sum_{s \in \mathbb{F}_p}\psi_a(s^2)$ is a quadratic Gauss sum over $\mathbb{F}_p$.
Considering the condition that $p \equiv 5\pmod{8}$, the following equation is obtained by Theorems 1.1.3 and 1.2.4 in \cite{BEW98}$\colon$
\begin{align}
\label{eq:Gauss2}
    \sum_{s \in \mathbb{F}_p}\psi_a(s^2)
    =\Bigl(\frac{a}{p}\Bigr)\sqrt{p}
    =\begin{cases}
    %\frac{q-1}{2} & \text{$a=0$}; \\
    \sqrt{p} & \text{$a \in C^{(2)}_0$};\\
    -\sqrt{p} & \text{$a \in C^{(2)}_1$},
  \end{cases}
\end{align}
where $(\frac{\cdot}{p})$ denotes the Legendre symbol.
By (\ref{eq:eigen1}), (\ref{eq:eigen2}) and (\ref{eq:Gauss2}), it holds that 
\begin{align}
\label{eq:eigen3}
     |\lambda_a|
     =\begin{cases}
    \frac{\sqrt{p}-1}{2} & \text{$a \in C^{(2)}_0$};\\
     \\
    \frac{\sqrt{p}+1}{2} & \text{$a \in C^{(2)}_1$}.
  \end{cases}
\end{align}

Since $Cay(\mathbb{F}_p^+, D)$ is a regular tournament, its adjacency matrix $A$ is normal and $A+A^T=J_p-I_p$, where $J_p$ and $I_p$ are the all-one and identity matrix of order $p$, respectively.
Thus the Perron-Frobenius theorem (e.g. \cite[Theorem 31.11]{LW01}) implies that for each $a \in\mathbb{F}_p^*$,
\begin{equation}
\label{eq:real1}
    \lambda_a+\overline{\lambda_a}=-1,
\end{equation}
and so,
\begin{equation*}
    {\rm Re}(\lambda_a)={\rm Re}(\overline{\lambda_a})=-\frac{1}{2}.
\end{equation*}
Since $\lambda_{-a}=\overline{\lambda_a}$ is also an eigenvalue and $-1\in C^{(2)}_0$ (because $p \equiv 1 \pmod{4}$), it holds by (\ref{eq:eigen3}) that
\begin{equation}
\label{eq:real2}
    |\lambda_a|^2={\rm Re}(\lambda_a)^2+{\rm Im}(\lambda_a)^2
     =\begin{cases}
    \frac{p-2\sqrt{p}+1}{4} & \text{$a \in C^{(2)}_0$};\\
     \\
    \frac{p+2\sqrt{p}+1}{4} & \text{$a \in C^{(2)}_1$}.
    \end{cases}
\end{equation}
By (\ref{eq:real1}) and (\ref{eq:real2}), 
\begin{equation}
    {\rm Im}(\lambda_a)^2
    =\begin{cases}
    \frac{p-2\sqrt{p}}{4} & \text{$a \in C^{(2)}_0$};\\
     \\
    \frac{p+2\sqrt{p}}{4} & \text{$a \in C^{(2)}_1$}.
    \end{cases}
\end{equation}

Now let $A_{>0}=\{a \in \mathbb{F}_p^* \mid {\rm Im}(\lambda_a)>0\}$ and $A_{<0}=\{a \in \mathbb{F}_p^* \mid {\rm Im}(\lambda_a)<0\}$.
Notice that $A_{<0}=-A_{>0}$.
Since $(A_{>0}, A_{<0})$ is a partition of $\mathbb{F}_p^*$, it holds that
\begin{equation}
\label{eq:eigenfinal}
    \lambda_a
    =\begin{cases}
    \frac{-1+i\sqrt{p-2\sqrt{p}}}{2} & \text{$a \in A_{>0} \cap C^{(2)}_0$};\\
     \\
    \frac{-1-i\sqrt{p-2\sqrt{p}}}{2} & \text{$a \in A_{<0} \cap C^{(2)}_0$};\\
    \\
    \frac{-1+i\sqrt{p+2\sqrt{p}}}{2} & \text{$a \in A_{>0} \cap C^{(2)}_1$};\\
     \\
    \frac{-1-i\sqrt{p+2\sqrt{p}}}{2} & \text{$a \in A_{<0} \cap C^{(2)}_1$}.
    \end{cases}
\end{equation}

Since $|C^{(2)}_0|=|C^{(2)}_1|=|A_{>0}|=|A_{<0}|=(p-1)/2$ and $\sum_{a \in \mathbb{F}_p}\lambda_a={\rm Trace}(A)=0$ (since $Cay(\mathbb{F}_p^+, D)$ has no loops) where $\lambda_0=(p-1)/2$, 
a simple computation shows that each of the four eigenvalues in (\ref{eq:eigenfinal}) have multiplicity $(p-1)/4$. 
\end{proof}

\begin{remark}
\label{rem:primepower}
We remark that $5^3=125$ is the unique prime power that is not a prime and satisfies the condition of Theorem~\ref{thm:DHL} (and hence Theorems~\ref{theorem-main1} and \ref{theorem-main2}). 
Indeed if $q \equiv 5 \pmod{8}$ is a power of an odd prime $p$, then $\log_p q$ must be odd. 
Except for $5^3=125$, there does not exist such prime powers with $\log_p q>1$.
This fact directly follows from a result by Cohn~\cite{C93} that if $r \geq 3$, the diophantine equation $X^2+4=Y^r$ has integer solutions if and only if $r=3$, where integer solutions $(X, Y)$ are only $(\pm 2, 2), (\pm 11, 5)$. 
% On the other hand, by Proposition~\ref{prop:ND}, $Cay(\mathbb{F}_{5^3}^+, C^{(4)}_0 \cup C^{(4)}_{1})$ is an $NDR_{125}$.
 \end{remark}

\begin{remark}
\label{rem-125}
The proof of Theorem~\ref{theorem-main1} can be used to show that $CT_{125}$ has the canonical spectrum.
Here a formula corresponding to (\ref{eq:Gauss2}) can be obtained by Theorems 1.1.3, 1.2.4 and 11.5.4 in \cite{BEW98}.
\end{remark}

\section{Concluding remarks}
\label{sect-rem}

\begin{enumerate}
    \item One can determine the full automorphism group ${\rm Aut}(CT_p)$ for primes $p \equiv 5 \pmod{8}$. Indeed, we have ${\rm Aut}(CT_p)=\{\sigma_{a, b} \mid a \in C_0^{(4)}, b \in \mathbb{F}_p \}$ where for $a \in \mathbb{F}_p^*$, $b \in \mathbb{F}_p$, the permutation $\sigma_{a, b}$ is defined as $x \mapsto ax+b$. The proof is similar to the discussion in \cite{G70}. It immediately follows from this result that $CT_p$ has exactly two arc orbits $\{(x,y) \mid x-y \in C^{(4)}_{0}\}$ and $\{(x,y) \mid x-y \in C^{(4)}_{1}\}$, which provides a new direct proof of (i) of Theorem 3.1 in \cite{AD92} (for prime cases).

    \item For a regular tournament $T$ with $n$ vertices and the second largest eigenvalue $\lambda(T)$ (in absolute value), we have (e.g. Remark~2 in \cite{S2021})
    $$
\lambda(T)\geq \frac{\sqrt{n+1}}{2}.
    $$
    Hence Theorem~\ref{theorem-main2} shows that $CT_q$ that is nearly-doubly-regular asymptotically minimizes the second largest eigenvalue with respect to the above bound. 
    By spectral characterizations of quasi-randomness, we see that a nearly-doubly-regular tournament $CT_q$ possesses (strong) quasi-randomness (see e.g. \cite{KS13} and \cite{S2021}). 

    \item Tabib~\cite{T80} showed that $Cay(\mathbb{Z}/{9}\mathbb{Z}, \{1, 2, 3, 5\})$ is the unique rotational $NDR_9$. It can be easily checked that $\{1, 2, 3, 5\}\subset \mathbb{Z}/{9}\mathbb{Z}$ is an $(9, 4, 1)$-almost difference set of the cyclic group $\mathbb{Z}/{9}\mathbb{Z}$ and hence $Cay(\mathbb{Z}/{9}\mathbb{Z}, \{1, 2, 3, 5\})$ is an $NDR_9$ by Lemma~\ref{lem-ND2}. We note that this is substantially obtained by a construction of complementary difference sets due to Szekeres~\cite{S69}. To see this, consider the finite field $\mathbb{F}_{19}$ and its multiplicative subgroup $C_0^{(2)}=\{1, 4, 5, 6, 7, 9, 11, 16, 17\} \cong \mathbb{Z}/9\mathbb{Z}$. Then by taking $2 \in \mathbb{F}_{19}^*$ as a primitive element, the subset $C_0^{(2)} \cap (C_0^{(2)}+1)=\{5, 6, 7, 17\} \subset \mathbb{F}_{19}$ yields a $(9, 4, 1)$-almost difference set of $\mathbb{Z}/9\mathbb{Z}$, namely $\{3, 5, 7, 8\} \subset \mathbb{Z}/9\mathbb{Z}$. One can check that $\{3, 5, 7, 8\}$ is mapped to $\{1, 2, 3, 5\}$ by the permutation $x \mapsto 4x \pmod{9}$ which is an automorphism of $\mathbb{Z}/9\mathbb{Z}$. %Note that for an abelian group $G$, if its subsets $D_1$ and $D_2$ satisfy that $D_2=(D_1)^{\tau}+g$ for some $\tau \in \Aut(G)$ and $g \in G$, then $Cay(G, D_1)$ and $Cay(G, D_2)$ are isomorphic. 
    Hence one can easily check that $Cay(\mathbb{Z}/{9}\mathbb{Z}, \{1, 2, 3, 5\})$ and $Cay(\mathbb{Z}/{9}\mathbb{Z}, \{3, 5, 7, 8\})$ are isomorphic.
\end{enumerate}

\section*{Acknowledgement}
We greatly appreciate S. V. Savchenko for sending us a paper~\cite{S16} and for his valuable comments and discussions.
We also deeply thank Koji Momihara for his many helpful comments.
This work was supported by JSPS KAKENHI Grant Numbers JP20J00469, JP23K13007.

\end{document}